%% file: Final_main2.tex
\documentclass[11pt,a4paper,reqno]{amsart}
\usepackage[british]{babel}
\usepackage{csquotes}
\usepackage{epsfig}
\usepackage{amsmath}
\usepackage{bbm}
\numberwithin{equation}{section}

\usepackage{amsfonts, amssymb, eucal, verbatim, amsthm, amscd, enumerate, mathtools}
\usepackage{framed}
\usepackage{enumitem}
\usepackage{cancel}
\usepackage[hidelinks]{hyperref}
\usepackage{ascmac}
\usepackage[left=2.5cm,right=2.5cm,top=2.5cm,bottom=2.5cm,footskip=1.5cm,headsep=1cm]{geometry}

\newcommand{\interior}[1]{%
  {\kern0pt#1}^{\mathrm{o}}%
}

\input{comands}

\newcommand{\ThetaPar}{\Theta}
\newcommand{\Znorm}[2]{Z_{#1,#2}}
\newcommand{\CV}{\mathbf{C}_V}

\newcommand{\cW}{\mathbf{c}_W}
\newcommand{\CVp}[1]{\mathbf{C}_{V,#1}}
\newcommand{\CWp}[1]{\mathbf{C}_{W,#1}}
\newcommand{\cVp}[1]{\mathbf{c}_{V,#1}}
\newcommand{\cWp}[1]{\mathbf{c}_{W,#1}}

\title[]{A dimension-free interpolation\\ of Caffarelli’s contraction theorem}

\author[B. Ammari]{Bader Ammari}
\address[]{Department of Mathematics, ETH Z\"{u}rich, R\"{a}mistrasse 101, CH-8092 Z\"{u}rich, Switzerland}
\email{bader.ammari@math.ethz.ch}

\author[A. Figalli]{Alessio Figalli}
\address[]{Department of Mathematics, ETH Z\"{u}rich, R\"{a}mistrasse 101, CH-8092 Z\"{u}rich, Switzerland}
\email{alessio.figalli@math.ethz.ch}

\begin{document}

\begin{abstract}
We prove global Lipschitz estimates for Brenier maps between probability measures on $\mathbb{R}^n$ whose densities belong to the family
\[
    \rho_{U,p}=\Znorm{U}{p}^{-1}\exp(-\ThetaPar_p(U)), \qquad
    \ThetaPar_p(t)=p\log\Bigl(1+\frac{t}{p}\Bigr),
    \qquad p\in[n,+\infty],
\]
with finite normalization constant $\Znorm{U}{p}$, and with the convention $\ThetaPar_{\infty}(t)=t$. 
We allow different parameters for source and target, $d,D\in[n,+\infty]$, with $d\le D$. Our global estimate is uniform in $n,d,D$,  and in the case $d=D<+\infty$, it improves the bounds of \cite{Alessio_Carlier_Santa} by removing their exponential dependence on the dimension. We also prove localized estimates inside fixed balls $B_R$ whose constants are stable under the limits $d,D\to+\infty$ and they allow us to recover Caffarelli's celebrated contraction theorem with sharp constants. 
\end{abstract}

\maketitle

\vspace{0.2cm}

\vspace{0.2cm}
\vspace{0.5cm}

\section{Introduction}

The $L^2$ optimal transport (or Brenier-McCann) map between two absolutely continuous probability measures $\mu=f(x)\, dx$ and $\nu=g(y)\, dy$ is the unique solution of the Monge problem for quadratic cost:
\[
    \min\left\{ \int_{\mathbb{R}^n}|x-T(x)|^2 f(x)\, dx : T_{\#}\mu=\nu \right\}.
\]
Here $T_{\#}\mu=\nu$ means that $\mu(T^{-1}(A))=\nu(A)$ for every Borel set $A\subset\mathbb{R}^n$.
By the Brenier-McCann theorem (see \cite{Brenier,McCann}), the solution of this quadratic optimal transport problem is the gradient of a convex function, $T=\nabla\varphi$. This mapping is a fundamental object in optimal transportation theory and has found applications in numerous fields such as functional and geometric inequalities, probability, and machine learning (see \cite{Figalli_Maggi_Pratelli,Ambrosio_Stra,Benamou_Carlier}).

In \cite{Caff}, Caffarelli proved that if $f=e^{-V}$ and $g=e^{-W}$ for some $V,W \in C_{\mathrm{loc}}^{1,1}(\mathbb{R}^n)$ with
\[
    D^2V \le \CV^{(2)}\,\mathrm{Id}, \qquad D^2W \ge \cW^{(2)}\,\mathrm{Id}
\]
almost everywhere, where $0<\cW^{(2)},\CV^{(2)}<+\infty$, then the Brenier-McCann map $T$ pushing $f\,dx$ forward to $g\,dy$ is globally Lipschitz with
\begin{equation}\label{Original_Caff_contraction}
    \|DT\|_{L^\infty(\mathbb{R}^n)}\le \sqrt{\CV^{(2)}/\cW^{(2)}}.
\end{equation}
Caffarelli's contraction theorem is notable for its dimension-free Lipschitz bound \cref{Original_Caff_contraction}. This dimension independence allows one to transfer functional inequalities, such as log-Sobolev and Poincaré inequalities, via the Lipschitz transport map between the two measures. For further references, extensions, alternative proofs, and applications, we refer to \cite{Kolesnikov,Kim_Milman,Jhaveri_Colombo,Vald,Colombo_Fathi,DePHilippis_Shenfeld}.

The purpose of this paper is to obtain analogous estimates for a family of probability measures interpolating between polynomial-type densities and log-concave densities. Our global estimate is uniform not only in the dimension $n$, but also in the interpolation parameters. For $p\in[n,+\infty)$ and a potential $U>-p$, set
\[
    \ThetaPar_p(U):=p\log\Bigl(1+\frac{U}{p}\Bigr), \qquad
    \Znorm{U}{p}:=\int_{\mathbb{R}^n}e^{-\ThetaPar_p(U)}\,dx,
\]
whenever the integral is finite, and define
\[
    \rho_{U,p}:=\Znorm{U}{p}^{-1}e^{-\ThetaPar_p(U)}
    =\Znorm{U}{p}^{-1}\Bigl(1+\frac{U}{p}\Bigr)^{-p}.
\]
We also set $\ThetaPar_{\infty}(U)=U$, $\Znorm{U}{\infty}:=\int_{\mathbb{R}^n}e^{-U}\,dx$, and $\rho_{U,\infty}=\Znorm{U}{\infty}^{-1}e^{-U}$. Since $\ThetaPar_p(t)\to t$ as $p\to+\infty$, the family $\rho_{U,p}$ interpolates between the densities considered in \cite{Alessio_Carlier_Santa,Fathi} and the log-concave densities appearing in Caffarelli's theorem.

For $p\in[n,+\infty)$, we use the global shorthand
\[
\begin{aligned}
\CVp{p}^{(1)}&:=\Biggl\| \frac{|\nabla V|}{\sqrt p+|\cdot|}\Biggr\|_{L^\infty(\mathbb{R}^n)}^2,
\qquad&
\CWp{p}^{(1)}&:=\Biggl\| \frac{|\nabla W|}{\sqrt p+|\cdot|}\Biggr\|_{L^\infty(\mathbb{R}^n)}^2,\\[1.2em]
\cVp{p}^{(0)}&:=\Biggl\|\frac{p+|\cdot|^2}{p+V}\Biggr\|_{L^\infty(\mathbb{R}^n)}^{-1},
\qquad&
\cWp{p}^{(0)}&:=\Biggl\|\frac{p+|\cdot|^2}{p+W}\Biggr\|_{L^\infty(\mathbb{R}^n)}^{-1},\\[1.2em]
\CVp{p}^{(0)}&:=\Biggl\|\frac{p+V}{p+|\cdot|^2}\Biggr\|_{L^\infty(\mathbb{R}^n)},
\qquad&
\CWp{p}^{(0)}&:=\Biggl\|\frac{p+W}{p+|\cdot|^2}\Biggr\|_{L^\infty(\mathbb{R}^n)}.
\end{aligned}
\]
When a local bound is needed, we write
\[
\begin{aligned}
\CWp{p}^{(1)}(R)&:=\Biggl\| \frac{|\nabla W|}{\sqrt p+|\cdot|}\Biggr\|_{L^\infty(B_R(0))}^2,
\qquad&
\cVp{p}^{(0)}(R)&:=\Biggl\|\frac{p+|\cdot|^2}{p+V}\Biggr\|_{L^\infty(B_R(0))}^{-1},\\[1.2em]
\CVp{p}^{(0)}(R)&:=\Biggl\|\frac{p+V}{p+|\cdot|^2}\Biggr\|_{L^\infty(B_R(0))},
\qquad&
\CWp{p}^{(0)}(R)&:=\Biggl\|\frac{p+W}{p+|\cdot|^2}\Biggr\|_{L^\infty(B_R(0))}.
\end{aligned}
\]
Here and below, we use the convention $B_{+\infty}(0)=\mathbb{R}^n$, and omit the argument when $R=+\infty$.
At the endpoint $p=+\infty$, for $R<+\infty$, we use the conventions
\[
\begin{gathered}
\cVp{\infty}^{(0)}(R)=\CVp{\infty}^{(0)}(R)
=\cWp{\infty}^{(0)}(R)=\CWp{\infty}^{(0)}(R)=1,\\
\CVp{\infty}^{(1)}=\CWp{\infty}^{(1)}=0.
\end{gathered}
\]
Inequalities such as $U>-p$ are understood as void when $p=+\infty$.

We shall allow different source and target parameters, denoted by $d$ and $D$. More precisely, we will study the monotone transport map from $\rho_{V,d}$ to $\rho_{W,D}$, where
\[
    d,D\in[n,+\infty], \qquad d\le D.
\]
With this notation, the polynomial-to-polynomial case studied in \cite{Alessio_Carlier_Santa,Fathi} corresponds to $D=d<+\infty$, while the case $D=d=+\infty$ corresponds to Caffarelli's setup.

\begin{remark}
The lower bound $n\le d$ is natural for the maximum-principle argument: after taking the Monge-Amp\`ere equation to the power $1/d$, the concavity of $\det^{1/d}$ is available.

The restriction $d\le D$ is also natural from the point of view of tails. Indeed, if $n\le D<d<+\infty$ and $V(x)=W(x)=|x|^2$, then the radial Brenier map $T(x)=t(|x|)x/|x|$ from $\rho_{V,d}$ to $\rho_{W,D}$ satisfies
\[
    \Znorm{W}{D}^{-1}\int_{t(r)}^{+\infty}\frac{s^{n-1}}{(1+s^2/D)^D}\,ds
    =
    \Znorm{V}{d}^{-1}\int_r^{+\infty}\frac{s^{n-1}}{(1+s^2/d)^d}\,ds.
\]
Thus, $t(r)\simeq r^{(2d-n)/(2D-n)}$ as $r\to+\infty$, which is superlinear when $d>D$. In particular, in this range, no global Lipschitz bound can hold in general for $T$.
\end{remark}

We first state a global Hessian estimate for the polynomial-to-polynomial case considered in \cite{Alessio_Carlier_Santa}. Its constant is uniform in $n,d,D$. To keep the statement readable, we collect the following global structural quantities, understood in the extended sense; if they make $C_{\mathrm{glob}}=+\infty$, then the estimate is void:
\[
\begin{aligned}
\mathfrak q_V&:=
\frac{\max\{1,\CVp{n}^{(0)}\}}{\min\{1,\cVp{n}^{(0)}\}},
\qquad&
\mathfrak q_W&:=
\frac{\max\{1,\CWp{n}^{(0)}\}}{\min\{1,\cWp{n}^{(0)}\}},\\[0.8em]
\mathfrak c_V&:=\min\{1,\cVp{n}^{(0)}\},
\qquad&
\mathfrak C_W&:=\max\{1,\CWp{n}^{(0)}\},\\[0.8em]
\mathfrak L_V&:=\CVp{n}^{(1)},
\qquad&
\mathfrak L_W&:=\CWp{n}^{(1)}.
\end{aligned}
\]
These quantities control all larger parameters. Indeed, for fixed $x$, the ratio
\[
    \frac{p+U(x)}{p+|x|^2}
\]
moves monotonically towards $1$ as $p$ increases, and $(\sqrt p+|x|)^{-1}$ is decreasing in $p$. Thus, the zeroth-order upper and lower constants are controlled by their values at $p=n$ and by the endpoint value $1$, while the first-order constants are largest at $p=n$. The ratios $\CVp{p}^{(0)}/\cVp{p}^{(0)}$ and $\CWp{p}^{(0)}/\cWp{p}^{(0)}$ need not be monotone in $p$, but they are bounded by $\mathfrak q_V$ and $\mathfrak q_W$.
Set
\[
\mathbf M_{\mathrm{glob}}:= 10^6\,\mathfrak q_V^2\mathfrak q_W^2,\qquad
A_{\mathrm{glob}}
:=
\frac{\CV^{(2)}\mathfrak C_W\mathbf M_{\mathrm{glob}}}
{\cW^{(2)}\mathfrak c_V},\qquad 
B_{\mathrm{glob}}
:=
\frac{8\mathfrak L_V\mathfrak L_W\mathbf M_{\mathrm{glob}}}
{(\cW^{(2)})^2\mathfrak c_V^2},
\]
and
\[
    C_{\mathrm{glob}}
    :=
    \sqrt{A_{\mathrm{glob}}+B_{\mathrm{glob}}}
    +\sqrt{B_{\mathrm{glob}}}.
\]

\begin{theorem}[Global dimension-free Hessian estimate]\label{Intro_global_theorem}
Let $d,D\in[n,+\infty]$ with $d\le D$. Let $V>-d$ and $W>-D$, assume that the normalizing constants are finite, and let $\nabla\varphi$ be the Brenier map pushing $\rho_{V,d}$ forward to $\rho_{W,D}$. Assume that
\[
    D^2V\le \CV^{(2)}\,\mathrm{Id}, \qquad
    D^2W\ge \cW^{(2)}\,\mathrm{Id},
\]
where $0<\cW^{(2)},\CV^{(2)}<+\infty$.
Then
\[
    \|D^2\varphi\|_{L^\infty(\mathbb{R}^n)}
    \le
    C_{\mathrm{glob}}.
\]
\end{theorem}

Taking $D=d<+\infty$ in \cref{Intro_global_theorem} gives the polynomial-to-polynomial estimate of \cite{Alessio_Carlier_Santa} in a form that is uniform in the parameters $n,d,D$, contrary to the result in \cite{Alessio_Carlier_Santa} which had an exponential dependence in $d$.

While this theorem is extremely robust, it does not recover the sharp Caffarelli bound \eqref{Original_Caff_contraction} in the case $d=D=+\infty$. To recover such a sharp result, we need first to prove quantitative bounds on the Hessian of $\varphi$ in a fixed ball $B_R(0)$, then we will let $D\to+\infty$ and $d\to+\infty$ at first and finally let $R\to+\infty$.

To prove this refined result, we need to introduce the quantities $\mathcal R_{d,D}(R)$ and $\mathcal G_{d,D}(R)$. These are finite-parameter local growth constants coming from Fathi's argument in \cite{Fathi}.

For $d,D\in[n,+\infty]$ with $d\le D$, set
\begin{equation}\label{Gamma_definition}
\Gamma_{d,D}:=
\begin{cases}
\displaystyle
+\infty, & D=d<+\infty,\\[0.8em]
\displaystyle
\left(\frac{2d-D}{D-d}\right)_+, & d<D<+\infty,\\[0.8em]
0, & D=+\infty.
\end{cases}
\end{equation}

For $d\le D<+\infty$ and $0<R<+\infty$, set
\[
    s_{d,R}:=\max\{R,\sqrt d\}
\]
and
\[
    \mathfrak m_{d,R}:=
    \frac{\omega_n(3s_{d,R})^n}{\Znorm{V}{d}\bigl(\CVp{d}^{(0)}(10s_{d,R})\bigr)^d}
    \left(1+\frac{100s_{d,R}^2}{d}\right)^{-d}.
\]
Here, $\omega_n$ denotes the volume of the unit ball in $\mathbb{R}^n$.
Also, define
\[
    \Psi_{W,D}(r):=
    \Znorm{W}{D}^{-1}\int_{B_r(0)^c}e^{-\ThetaPar_D(W(y))}\,dy,
\]
and
\begin{equation}\label{eq:RG}
    \mathcal R_{d,D}(R):=
    3\inf\{r\ge0:\Psi_{W,D}(r)\le \mathfrak m_{d,R}\},\qquad
    \mathcal G_{d,D}(R):=1+\frac{\mathcal R_{d,D}(R)^2}{D}.
\end{equation}
Since $\mathfrak m_{d,R}>0$ and $\Psi_{W,D}(r)\to0$ as $r\to+\infty$, the radius $\mathcal R_{d,D}(R)$ is finite.

For the localized estimate, with $\Gamma_{d,D}$ as in \cref{Gamma_definition}, define
\begin{equation}\label{Local_factors_definition}
\begin{aligned}
\Lambda_{d,D}(R)
&:=
\begin{cases}
\displaystyle
\CWp{D}^{(0)}(\mathcal R_{d,D}(R))\mathcal G_{d,D}(R), & D<+\infty,\\[0.8em]
1, & D=+\infty,
\end{cases}\\[1.2em]
\Xi_{d,D}(R)
&:=
\begin{cases}
\displaystyle
\min\left\{
\frac{D}{d}\,
\frac{\CWp{D}^{(1)}(\mathcal R_{d,D}(R))}
{\CWp{D}^{(0)}(\mathcal R_{d,D}(R))},
\cW^{(2)}\Gamma_{d,D}
\right\},
& D<+\infty,\\[0.8em]
0, & D=+\infty.
\end{cases}
\end{aligned}
\end{equation}
Finally, set
\[
\begin{aligned}
A_{V,W}^{d,D}(R)
&:=
\frac{\CV^{(2)}\Lambda_{d,D}(R)}
{\cW^{(2)}\cVp{d}^{(0)}(R)},
\qquad&
B_{V,W}^{d,D}(R)
&:=
\frac{4\CVp{d}^{(1)}\Lambda_{d,D}(R)\Xi_{d,D}(R)}
{(\cW^{(2)})^2(\cVp{d}^{(0)}(R))^2}.
\end{aligned}
\]

\begin{theorem}[Localized Hessian estimate]\label{Intro_local_theorem}
Let $d,D\in[n,+\infty]$ with $d\le D$, and let $0<R<+\infty$. Let $V>-d$ and $W>-D$, assume that the normalizing constants are finite, and let $\nabla\varphi$ be the Brenier map pushing $\rho_{V,d}$ forward to $\rho_{W,D}$. Assume that
\[
    D^2V\le \CV^{(2)}\,\mathrm{Id}, \qquad
    D^2W\ge \cW^{(2)}\,\mathrm{Id},
\]
where $0<\cW^{(2)},\CV^{(2)}<+\infty$. With the constants just defined,
\[
    \|D^2\varphi\|_{L^\infty(B_R(0))}
    \le
    \sqrt{A_{V,W}^{d,D}(R)+B_{V,W}^{d,D}(R)}+\sqrt{B_{V,W}^{d,D}(R)}.
\]
\end{theorem}

\hfill \\
For fixed $0<R<+\infty$ and $d<+\infty$, the constants in \cref{Intro_local_theorem} with $D<+\infty$ converge to the choices $\Lambda_{d,\infty}(R)=1$ and $\Xi_{d,\infty}(R)=0$ in \cref{Local_factors_definition}. Indeed, the convergence and tightness of $\rho_{W,D}$ as $D\to+\infty$ keep the Fathi's radius $\mathcal R_{d,D}(R)$ bounded, while $\mathcal G_{d,D}(R)\to1$, $\CWp{D}^{(0)}(\mathcal R_{d,D}(R))\to1$, and $\Xi_{d,D}(R)=0$ once $D\ge2d$. 

As we now show, to obtain the sharp Caffarelli limit, the order of limits is important: one first keeps $R<+\infty$ fixed and lets $D\to+\infty$ and then $d\to+\infty$, and only afterwards sends $R\to+\infty$. In the parameterization $R=\rho\sqrt d$, this means that $\rho=R/\sqrt d$ tends to zero during the parameter limit.

The polynomial-to-log-concave case $D=+\infty$ gives a direct interpolation between the polynomial estimates and Caffarelli's theorem. In this case, the growth factor disappears from the bound.

\begin{corollary}\label{Intro_poly_log_corollary}
Under the assumptions of \cref{Intro_local_theorem} with $D=+\infty$, for every $0<R<+\infty$ one has
\begin{equation}
    \label{eq:local poly log}
      \|D^2\varphi\|_{L^\infty(B_R(0))}
    \le
    \sqrt{\frac{\CV^{(2)}}{\cW^{(2)}\cVp{d}^{(0)}(R)}}.
\end{equation}
In particular,
\[
    \|D^2\varphi\|_{L^\infty(\R^n)}
    \le
    \sqrt{\frac{\CV^{(2)}}{\cW^{(2)}\cVp{d}^{(0)}}}.
\]

\end{corollary}

Starting from \eqref{eq:local poly log}, passing first to $d\to+\infty$ for fixed $R<+\infty$, and then to $R\to+\infty$, allows us to recover Caffarelli's contraction theorem.

\begin{corollary}\label{Intro_caffarelli_corollary}
If $f=e^{-V}$ and $g=e^{-W}$ are probability densities, with $V,W\in C_{\mathrm{loc}}^{1,1}(\mathbb{R}^n)$ and
\[
    D^2V\le \CV^{(2)}\,\mathrm{Id}, \qquad
    D^2W\ge \cW^{(2)}\,\mathrm{Id},
\]
then the Brenier map $T$ pushing $f\,dx$ forward to $g\,dy$ satisfies
\[
    \|DT\|_{L^\infty(\mathbb{R}^n)}
    \le
    \sqrt{\CV^{(2)}/\cW^{(2)}}.
\]
\end{corollary}

The approaches of \cite{Alessio_Carlier_Santa} and \cite{Colombo_Fathi}, which inspired \cite{Fathi}, are distinct: the latter relies on monotonicity of the optimal transport and concentration inequalities to deduce growth estimates, while the former uses the Monge-Amp\`ere equation and a maximum-principle argument. We combine these two viewpoints. After deriving the two-parameter maximum principle, we prove a local finite-parameter growth estimate by Fathi's argument and insert this explicit bound into the localized maximum-principle estimate. The endpoint $D=+\infty$ is obtained from the corresponding endpoint differential inequality, and the constants above are arranged so that the local finite-$D$ estimates converge to it for fixed $R<+\infty$.

\section[Proofs of the main estimates]{Proofs of \texorpdfstring{\cref{Intro_global_theorem} and \cref{Intro_local_theorem}}{the main estimates}}

We first prove the estimates under the additional assumption that the potentials and the Brenier potential are smooth, and that the second directional derivative whose supremum gives $\|D^2\varphi\|_{L^\infty}$ attains its maximum. This allows us to write the maximum-principle computation at the level of second derivatives. Then, at the end of the paper, we recall the argument used in \cite{Alessio_Carlier_Santa} to circumvent these issues.

We begin with the finite-parameter case $n\le d\le D<+\infty$, assume that $V>-d$ and $W>-D$, and let $\nabla\varphi$ be the Brenier map pushing $\rho_{V,d}$ forward to $\rho_{W,D}$.

Recalling that $\ThetaPar_p(t)=p\log\left(1+\frac{t}{p}\right)$ and $\ThetaPar_\infty(t)=t$, we have
\[
   \ThetaPar_p'(t)=\frac{p}{p+t}, \quad
    \ThetaPar_p''(t)=-\frac{p}{(p+t)^2},\qquad \ThetaPar_\infty'(t)=1,\quad \ThetaPar_\infty''(t)=0.
\]
The Monge-Amp\`ere equation reads
\[
    \Znorm{V}{d}^{-1}e^{-\ThetaPar_d(V)}
    =
    \Znorm{W}{D}^{-1}e^{-\ThetaPar_D(W\circ\nabla\varphi)}
    \det(D^2\varphi).
\]
Setting $F=\det^{1/d}$, we can rewrite it as
\begin{equation}\label{MA_two_parameter}
    \Bigl(\frac{\Znorm{V}{d}}{\Znorm{W}{D}}\Bigr)^{1/d}
    F(D^2\varphi)
    =
    \exp\left(\frac{\ThetaPar_D(W\circ\nabla\varphi)-\ThetaPar_d(V)}{d}\right).
\end{equation}

\begin{lemma}\label{Two_parameter_differential_inequality}
Assume that $V,W,\varphi$ are smooth and that the function
\[
    \mathbb{R}^n\times \mathbb{S}^{n-1}\ni (x,e)\mapsto \langle D^2\varphi(x)e,e\rangle
\]
attains its maximum at $(\overline x,e_1)$. Set $\overline y=\nabla\varphi(\overline x)$ and $\lambda=\varphi_{11}(\overline x)$. Then
\begin{equation}\label{Exact_second_variation}
\begin{aligned}
0 \ge{}&
\ThetaPar_D'(W)W_{11}\lambda^2-\ThetaPar_d'(V)V_{11}
-2\ThetaPar_d''(V)V_1^2 \\
&+\left(\frac{\ThetaPar_D'(W)^2}{d}+\ThetaPar_D''(W)\right)W_1^2\lambda^2
-\frac{2\ThetaPar_d'(V)\ThetaPar_D'(W)}{d}V_1W_1\lambda,
\end{aligned}
\end{equation}
where $V$ and its derivatives are evaluated at $\overline x$, while $W$ and its derivatives are evaluated at $\overline y$.

In particular, since $D\ge d$, if $W_{11}>0$ then
\begin{equation}\label{Master_second_variation}
    \begin{aligned}
    W_{11}\lambda^2
    \le{}&
    \frac{1}{1-\epsilon}\frac{d}{D}\frac{D+W}{d+V}V_{11} \\
    &+
    \frac{1}{\epsilon(1-\epsilon)}
    \frac{V_1^2W_1^2}{W_{11}(d+V)^2},
    \end{aligned}
\end{equation}
for every $\epsilon\in(0,1)$.
\end{lemma}

\begin{proof}
At the maximum point, $e_1$ is an eigenvector of $D^2\varphi(\overline x)$, and hence
\[
    \nabla\varphi_{11}(\overline x)=0,\qquad
    D^2\varphi_{11}(\overline x)\le 0,\qquad
    \varphi_{i1}(\overline x)=0 \quad \text{for } i\neq 1.
\]
Let
\[
    H(x):=\frac{\ThetaPar_D(W(\nabla\varphi(x)))-\ThetaPar_d(V(x))}{d}.
\]
By \cref{MA_two_parameter}, the right-hand side is $e^H$. Differentiating twice in the $e_1$ direction and using the concavity and monotonicity of $F=\det^{1/d}$ give
\[
    0\ge (e^H)_{11}=e^H(H_{11}+H_1^2).
\]
At $\overline x$, using $\varphi_{i1}=0$ for $i\neq 1$ and $\varphi_{i11}=0$ for all $i$, this becomes exactly \cref{Exact_second_variation}.

Since
\[
    \frac{\ThetaPar_D'(W)^2}{d}+\ThetaPar_D''(W)
    =
    \frac{D(D-d)}{d(D+W)^2}\ge 0,
\]
we may drop this term, as well as the nonnegative $V_1^2$ term, to get
\[
    \ThetaPar_D'(W)W_{11}\lambda^2
    \le
    \ThetaPar_d'(V)V_{11}
    +\frac{2\ThetaPar_d'(V)\ThetaPar_D'(W)}{d}V_1W_1\lambda.
\]
After division by $\ThetaPar_D'(W)$ and using $\ThetaPar_d'(V)=d/(d+V)$, Young's inequality gives \cref{Master_second_variation}.
\end{proof}

We first prove the growth estimates needed to use \cref{Master_second_variation}. 
The next proposition is a slight reformulation of Fathi's monotonicity estimate \cite[Theorem 2.3]{Fathi}. We recall that $\mathcal R_{d,D}(R)$ and $\mathcal G_{d,D}(R)$ are defined as in  
\eqref{eq:RG}.

\begin{proposition}\label{Local_D_growth}
Let $n\le d\le D<+\infty$ and $0<R<+\infty$, and let $\nabla\varphi$ be the optimal transport map pushing $\rho_{V,d}$ forward to $\rho_{W,D}$. 
Then
\[
    |\nabla\varphi(x)|\le \mathcal R_{d,D}(R)
    \qquad \text{for a.e. }\, x\in B_R(0),
\]
and
\[
    \frac{d}{D}\frac{D+|\nabla\varphi(x)|^2}{d+|x|^2}
    \le
    \mathcal G_{d,D}(R)
    \qquad \text{for a.e. }\, x\in B_R(0).
\]
\end{proposition}

\begin{proof}
Let $\mu_{V,d}$ and $\mu_{W,D}$ denote the measures with densities $\rho_{V,d}$ and $\rho_{W,D}$, respectively, and let $s=s_{d,R}$. For $x\in B_R(0)$ and $u\in\mathbb{S}^{n-1}$, every $z\in B_{3s}(x+6su)$ satisfies $|z|\le10s$. Hence
\[
\mu_{V,d}\bigl(B_{3s}(x+6su)\bigr)
\ge
\frac{1}{\Znorm{V}{d}\bigl(\CVp{d}^{(0)}(10s)\bigr)^d}
\int_{B_{3s}(x+6su)}
\Bigl(1+\frac{|z|^2}{d}\Bigr)^{-d}\,dz \ge \mathfrak m_{d,R}.
\]
Moreover,
\[
    \mu_{W,D}\bigl(B_r(0)^c\bigr)= \Psi_{W,D}(r)
    \qquad \forall\, r\ge0.
\]
If $\nabla\varphi(x)=0$, there is nothing to prove. Otherwise, take $u=\nabla\varphi(x)/|\nabla\varphi(x)|$. The monotonicity of $\nabla\varphi$ gives
\[
    \nabla\varphi\bigl(B_{3s}(x+6su)\bigr)
    \subset B_{|\nabla\varphi(x)|/3}(0)^c.
\]
Therefore,
\[
\begin{aligned}
    \mathfrak m_{d,R}
    &\le
    \mu_{V,d}\bigl(B_{3s}(x+6su)\bigr)=
    \mu_{W,D}\Bigl(\nabla\varphi\bigl(B_{3s}(x+6su)\bigr)\Bigr)\\
    &\le
    \mu_{W,D}\bigl(B_{|\nabla\varphi(x)|/3}(0)^c\bigr)
    =
    \Psi_{W,D}(|\nabla\varphi(x)|/3),
\end{aligned}
\]
which yields
\[
    |\nabla\varphi(x)|\le
    3\inf\{r\ge0:\Psi_{W,D}(r)\le \mathfrak m_{d,R}\}
    =
    \mathcal R_{d,D}(R).
\]
The second bound follows from $d+|x|^2\ge d$.
\end{proof}

We will also need another bound on the growth of the transport map, which follows again from Fathi's argument.
For $p\in[n,+\infty)$, let
\[
    I_p:=\int_{\mathbb{R}^n}\Bigl(1+\frac{|z|^2}{p}\Bigr)^{-p}\,dz.
\]
Also, for $d\le D<+\infty$, we set
\begin{equation}\label{Finite_growth_constants}
\begin{aligned}
\mathbf{K}_{d,D,n}
&:=
3\left[
\left(\frac{\CVp{d}^{(0)}}{\cVp{d}^{(0)}}\right)^d
\left(\frac{\CWp{D}^{(0)}}{\cWp{D}^{(0)}}\right)^D
\left(\frac{5}{4}\right)^d10^{2d}3^{-n}
\left(\frac{D}{d}\right)^D
\frac{I_d}{I_D}
\right]^{\frac{1}{2D-n}},\\
\mathbf{M}_{d,D,n}
&:=
\frac{d}{D}\mathbf{K}_{d,D,n}^2.
\end{aligned}
\end{equation}

\begin{proposition}\label{Finite_D_growth}
Let $n\le d\le D<+\infty$, and let $\nabla\varphi$ be the optimal transport map pushing $\rho_{V,d}$ forward to $\rho_{W,D}$. Then, with $\mathbf{K}_{d,D,n}$ and $\mathbf{M}_{d,D,n}$ as above,
\[
    |\nabla\varphi(x)|
    \le
    \mathbf{K}_{d,D,n}\max\{|x|,\sqrt d\}
    \qquad \text{for a.e. }\, x\in\mathbb{R}^n.
\]
In particular,
\[
    D+|\nabla\varphi(x)|^2
    \le
    \frac{D}{d}\left(1+\mathbf{M}_{d,D,n}\right)(d+|x|^2)
    \qquad \text{for a.e. }\, x\in\mathbb{R}^n.
\]
\end{proposition}

\begin{proof}
We follow the strategy of \cite[Lemma 3.6 (i)]{Fathi}. As before, $\mu_{V,d}$ and $\mu_{W,D}$ denote the measures with densities $\rho_{V,d}$ and $\rho_{W,D}$, respectively.

The definitions of the constants imply
\[
    \Znorm{V}{d}\le (\cVp{d}^{(0)})^{-d}I_d,
    \qquad
    \Znorm{W}{D}\ge (\CWp{D}^{(0)})^{-D}I_D.
\]
If $\nabla\varphi(x)=0$, there is nothing to prove; otherwise take $u=\nabla\varphi(x)/|\nabla\varphi(x)|$. Let $r_x:=\max\{|x|,\sqrt d\}$. As in the proof of \cite[Theorem 4.4]{Fathi}, every $z\in B_{3r_x}(x+6r_xu)$ satisfies
\[
    2r_x\le |z|\le 10r_x.
\]
Since $r_x\ge\sqrt d$, this gives
\[
    \frac{1}{1+\frac{|z|^2}{d}}\ge \frac{4d}{5|z|^2}.
\]
On the one hand, 
\[
\begin{aligned}
\mu_{V,d}(B_{3r_x}(x+6r_xu)) &\ge Z_{V,d}^{-1}(\CVp{d}^{(0)})^{-d} \int_{B_{3r_x}(x+6r_xu)}\Bigl(1+\frac {|z|^2}d\Bigl)^{-d}\, dz
\\
&\ge Z_{V,d}^{-1}(\mathbf{C}_{V,d}^{(0)})^{-d} \Bigl(\frac 45\Bigl)^dd^d \int_{B_{3r_x(x+6r_xu)}}|z|^{-2d}\, dz \\
&\ge
\left(\frac{\cVp{d}^{(0)}}{\CVp{d}^{(0)}}\right)^d
\Bigl(\frac{4}{5}\Bigr)^d
\frac{d^d r_x^{n-2d}}{I_d}
10^{-2d}3^n\omega_n\\
&\ge
\left(\frac{\cVp{d}^{(0)}}{\CVp{d}^{(0)}}\right)^d
\Bigl(\frac{4}{5}\Bigr)^d
\frac{d^D r_x^{n-2D}}{I_d}
10^{-2d}3^n\omega_n.
\end{aligned}
\]
On the other hand, 
\[
\begin{aligned}
\mu_{W,D}(|y|\ge r)
&\le
\left(\frac{\CWp{D}^{(0)}}{\cWp{D}^{(0)}}\right)^D
\frac{1}{I_D}
\int_{B_r(0)^c}\Bigl(1+\frac{|y|^2}{D}\Bigr)^{-D}\,dy \\
&\le \left(\frac{\CWp{D}^{(0)}}{\cWp{D}^{(0)}}\right)^D
\frac{1}{I_D} D^D \omega_n \int_{r}^{+\infty} t^{n-2D-1}\, dt \\
&\le
\left(\frac{\CWp{D}^{(0)}}{\cWp{D}^{(0)}}\right)^D
\frac{D^{D}\omega_n}{(2D-n)I_D}r^{n-2D}
=:\psi(r).
\end{aligned}
\]

By \cite[Theorem 2.3]{Fathi},
\begin{multline}
|\nabla\varphi(x)|
\le
3\psi^{-1}\bigl(\mu_{V,d}(B_{3r_x}(x+6r_xu))\bigr)\\
\le
3\left[
\left(\frac{\CVp{d}^{(0)}}{\cVp{d}^{(0)}}\right)^d
\left(\frac{\CWp{D}^{(0)}}{\cWp{D}^{(0)}}\right)^D
\left(\frac{5}{4}\right)^d10^{2d}3^{-n}
\left(\frac{D}{d}\right)^D 
\frac{I_d}{(2D-n)I_D}
\right]^{\frac{1}{2D-n}}
r_x \le
\mathbf{K}_{d,D,n}r_x.
\end{multline}
The final displayed growth bound follows immediately.
\end{proof}

To prove our Theorem~\ref{Intro_global_theorem}, we will use the following uniform bound:
\begin{equation}\label{Uniform_M_bound}
    1+\mathbf M_{d,D,n}\le \mathbf M_{\mathrm{glob}}
    \qquad \forall\, n\le d\le D<+\infty.
\end{equation}
To prove its validity, we observe that, by the monotonicity observation in the introduction,
\[
    \frac{\CVp{d}^{(0)}}{\cVp{d}^{(0)}}\le \mathfrak q_V,
    \qquad
    \frac{\CWp{D}^{(0)}}{\cWp{D}^{(0)}}\le \mathfrak q_W.
\]
In addition, the corresponding exponents are bounded by $2$, and
\[
    \frac dD\left(\frac Dd\right)^{\frac{2D}{2D-n}}\le 2.
\]
Furthermore, we claim that
\[
    9\left[\left(\frac{5}{4}\right)^d10^{2d}3^{-n}\right]^{\frac{2}{2D-n}}
    \le
    9\left[125^D3^{-n}\right]^{\frac{2}{2D-n}}
    \le
    15625.
\]
Indeed, writing $\tau=n/D$, the logarithm of the middle term is
\[
    \log 9+\frac{2\log 125-2\tau\log 3}{2-\tau},
\]
which increases for $0<\tau\le1$ and is maximal at $\tau=1$.
Moreover, the function $p\mapsto I_p$ decreases by pointwise monotonicity of the integrand, while
\[
    I_n\le 2\omega_n n^{n/2},
    \qquad
    I_D\ge e^{-n}\omega_n n^{n/2}.
\]
Thus,
\[
    \left(\frac{I_d}{I_D}\right)^{\frac{2}{2D-n}}
    \le 4e^2,
\]
and therefore
\[
    \mathbf M_{d,D,n}
    \le
    125000e^2\,\mathfrak q_V^2\mathfrak q_W^2
    <
    924000\, \mathfrak q_V^2\mathfrak q_W^2 ,
\]
which implies \cref{Uniform_M_bound}.

We next prove a maximum-principle estimate for the Hessian of $\varphi$ inside a ball $B_R$ under a growth bound on its gradient.

\begin{proposition}\label{Conditional_finite_D}
Let $n\le d\le D<+\infty$ and $0<R\le+\infty$. Assume that $V>-d$, $W>-D$, and
\[
    D^2V\le \CV^{(2)}\,\mathrm{Id}, \qquad
    D^2W\ge \cW^{(2)}\,\mathrm{Id},
\]
with $0<\cW^{(2)},\CV^{(2)}<+\infty$. Let $\nabla\varphi$ be the optimal transport map pushing $\rho_{V,d}$ forward to $\rho_{W,D}$. Assume that, for some $\mathsf R_R,\mathsf G_R\in[0,+\infty)$,
\begin{equation}\label{Growth_assumption_general}
    |\nabla\varphi(x)|\le \mathsf R_R,
    \qquad
    \frac{d}{D}\frac{D+|\nabla\varphi(x)|^2}{d+|x|^2}\le \mathsf G_R
    \qquad \text{for a.e. }\, x\in B_R(0).
\end{equation}
Then
\begin{equation}\label{Finite_D_local_bound}
    \|D^2\varphi\|_{L^\infty(B_R(0))}
    \le
    \sqrt{\widetilde A_{V,W}^{d,D}(R)+\widetilde B_{V,W}^{d,D}(R)}+\sqrt{\widetilde B_{V,W}^{d,D}(R)},
\end{equation}
where, with $\Gamma_{d,D}$ as in \cref{Gamma_definition}, we set
\[
\widetilde\Lambda_R
:=
\CWp{D}^{(0)}(\mathsf R_R)\mathsf G_R,\qquad
\widetilde\Xi_R
:=
\min\left\{
\frac{D}{d}\,
\frac{\CWp{D}^{(1)}(\mathsf R_R)}
{\CWp{D}^{(0)}(\mathsf R_R)},
\cW^{(2)}\Gamma_{d,D}
\right\}
\]
\[
\widetilde A_{V,W}^{d,D}(R)
=
    \frac{\CV^{(2)}\widetilde\Lambda_R}{\cW^{(2)}\cVp{d}^{(0)}(R)},\qquad
\widetilde B_{V,W}^{d,D}(R)
=
    \frac{4\CVp{d}^{(1)}\widetilde\Lambda_R\widetilde\Xi_R}
    {(\cW^{(2)})^2(\cVp{d}^{(0)}(R))^2}.
\]
\end{proposition}

\begin{proof}
As discussed above, we prove the result under the assumption that there exists an interior  maximum point $(\overline x,e_1)$ with $\overline x\in B_R(0)$. This assumption can be removed by using the localized incremental-ratio argument described at the end of the paper.

Set $\overline y=\nabla\varphi(\overline x)$. By \cref{Master_second_variation} and the Hessian bounds,
\[
\cW^{(2)}\varphi_{11}(\overline x)^2
\le
\frac{\CV^{(2)}}{1-\epsilon}
\frac{d(D+W(\overline y))}{D(d+V(\overline x))} 
+
\frac{1}{\epsilon(1-\epsilon)\cW^{(2)}}
\frac{V_1(\overline x)^2W_1(\overline y)^2}{(d+V(\overline x))^2}.
\]
Using the definitions of the constants and the growth assumption \cref{Growth_assumption_general}, we have
\[
    \frac{d}{D}\frac{D+W(\overline y)}{d+V(\overline x)}
    \le
    \frac{\CWp{D}^{(0)}(\mathsf R_R)\mathsf G_R}{\cVp{d}^{(0)}(R)}
\]
and
\[
    \frac{V_1(\overline x)^2W_1(\overline y)^2}{(d+V(\overline x))^2}
    \le
    \frac{4\CVp{d}^{(1)}\CWp{D}^{(1)}(\mathsf R_R)\frac{D}{d}\mathsf G_R}{(\cVp{d}^{(0)}(R))^2}.
\]
Therefore, with
\[
    B_y:=
    \frac{4\CVp{d}^{(1)}\CWp{D}^{(1)}(\mathsf R_R)\frac{D}{d}\mathsf G_R}
    {(\cW^{(2)})^2(\cVp{d}^{(0)}(R))^2},
\]
we have
\[
    \begin{aligned}
    \varphi_{11}(\overline x)^2
    \le{}&
    \frac{\widetilde A_{V,W}^{d,D}(R)}{1-\epsilon}
    +
    \frac{B_y}{\epsilon(1-\epsilon)}.
    \end{aligned}
\]
Optimizing in $\epsilon\in(0,1)$ gives the desired estimate with $B_y$ in place of $\widetilde B_{V,W}^{d,D}(R)$.

If $D>d$, we can instead use the positive terms $W_1^2\lambda^2$ and $V_1^2$ in \cref{Exact_second_variation}. Minimizing the resulting quadratic polynomial in $W_1\lambda$ gives
\[
    \ThetaPar_D'(W)W_{11}\lambda^2
    \le
    \ThetaPar_d'(V)V_{11}
    +
    \frac{d}{(d+V)^2}
    \left(\frac{2d-D}{D-d}\right)_+
    V_1^2.
\]

Using again the definitions of the constants and \cref{Growth_assumption_general}, this yields
\[
    \varphi_{11}(\overline x)^2
    \le
    \widetilde A_{V,W}^{d,D}(R)
    +
    B_c,
\]
where
\[
    B_c:=
    \frac{4\CVp{d}^{(1)}\CWp{D}^{(0)}(\mathsf R_R)\mathsf G_R\Gamma_{d,D}}
    {\cW^{(2)}(\cVp{d}^{(0)}(R))^2}.
\]
For $D=d$, we use the convention $B_c=+\infty$. Since the function
\[
    B\mapsto \sqrt{\widetilde A_{V,W}^{d,D}(R)+B}+\sqrt B
\]
is increasing, the two bounds give \cref{Finite_D_local_bound} with $B=\min\{B_y,B_c\}$, namely with $\widetilde B_{V,W}^{d,D}(R)$.
\end{proof}

We can now prove all the results stated in the introduction.

When $0<R<+\infty$, Proposition \ref{Local_D_growth} gives \cref{Growth_assumption_general} with
\[
    \mathsf R_R=\mathcal R_{d,D}(R),
    \qquad
    \mathsf G_R=\mathcal G_{d,D}(R).
\]
Thus, Proposition \ref{Conditional_finite_D} proves \cref{Intro_local_theorem} when $D<+\infty$. If $R=+\infty$, \cref{Finite_D_growth} gives \cref{Growth_assumption_general} with
\[
    \mathsf R_\infty=+\infty,
    \qquad
    \mathsf G_\infty=1+\mathbf{M}_{d,D}.
\]
Applying Proposition \ref{Conditional_finite_D} with $R=+\infty$ gives the sharper finite-parameter estimate
\[
    \|D^2\varphi\|_{L^\infty(\mathbb{R}^n)}
    \le
    \sqrt{A_{V,W}^{d,D}+B_{V,W}^{d,D}}+\sqrt{B_{V,W}^{d,D}},
\]
where
\[
\begin{aligned}
A_{V,W}^{d,D}
&:=
\frac{\CV^{(2)}\CWp{D}^{(0)}\, (1+\mathbf M_{d,D})}
{\cW^{(2)}\cVp{d}^{(0)}},\\
B_{V,W}^{d,D}
&:=
\frac{4\CVp{d}^{(1)}}{(\cW^{(2)})^2(\cVp{d}^{(0)})^2}
\min\left\{
\CWp{D}^{(1)}\frac{D}{d}(1+\mathbf M_{d,D}),
\cW^{(2)}\CWp{D}^{(0)}\, (1+\mathbf M_{d,D})\Gamma_{d,D}
\right\}.
\end{aligned}
\]
By \cref{Uniform_M_bound}, $A_{V,W}^{d,D}\le A_{\mathrm{glob}}$. Moreover, if $D<2d$ then $D/d\le2$, while if $D\ge2d$, the second term in the minimum above is zero. Hence, $B_{V,W}^{d,D}\le B_{\mathrm{glob}}$ and this proves \cref{Intro_global_theorem} when $D<+\infty$. By approximation, the result holds also in the limiting cases when $d$ or $D$ are equal to $+\infty$.

\medskip

For the endpoint case $D=+\infty$ and $d<+\infty$, the same second-variation computation, as in \cref{Exact_second_variation}, gives:
\[
\begin{aligned}
0&\ge W_{11}\lambda^2-\frac{d}{d+V}V_{11}
+\frac{2d}{(d+V)^2}V_1^2 +\frac{1}{d}W_1^2\lambda^2
-\frac{2}{d+V}V_1W_1\lambda\\
&=W_{11}\lambda^2-\frac{d}{d+V}V_{11}
+\frac{d}{(d+V)^2}V_1^2+\frac{1}{d}\left(W_1\lambda-\frac{d}{d+V}V_1\right)^2,
\end{aligned}
\]
so, in particular,
\begin{equation}\label{Endpoint_second_variation}
    \Bigl(1+\frac{V}{d}\Bigr)W_{11}\lambda^2\le V_{11}.
\end{equation}
At an interior maximizer $(\tilde x,e_1)$ with $\tilde x\in B_R(0)$, this gives
\[
    1+\frac{V(\tilde x)}{d}
    =
    \frac{d+V(\tilde x)}{d}
    \ge
    \cVp{d}^{(0)}(R)\frac{d+|\tilde x|^2}{d}
    \ge
    \cVp{d}^{(0)}(R),
\]
and therefore
\[
    \cVp{d}^{(0)}(R)\cW^{(2)}\varphi_{11}(\tilde x)^2
    \le
    \CV^{(2)}.
\]
This is the estimate in \cref{Intro_local_theorem} when $D=+\infty$,
and shows the validity of
\eqref{eq:local poly log}.
Letting first $R\to+\infty$ proves Corollary~\ref{Intro_poly_log_corollary}.
Taking instead first the limit as $d\to +\infty$ and then $R \to +\infty$, recovers Caffarelli's contraction theorem as stated in Corollary~\ref{Intro_caffarelli_corollary}.

\bigskip

Following \cite{Alessio_Carlier_Santa}, we  conclude this paper by  recalling the standard argument that turns the preceding smooth maximum-principle computation into a proof under the stated hypotheses. By regularizing $V$ and $W$, one may first assume that $V,W\in C_{\mathrm{loc}}^{2,\alpha}(\mathbb{R}^n)$, so that the Brenier potential is $C^4$ on the region under consideration. The estimates obtained above are stable under this regularization, and the constants converge back to the original ones as the regularizing parameter goes to zero.

For $\epsilon>0$ and $e\in\mathbb{S}^{n-1}$, set
\[
    \varphi^\epsilon(x,e)
    :=
    \varphi(x+\epsilon e)+\varphi(x-\epsilon e)-2\varphi(x).
\]
After the usual compact truncation of the target measure, $\varphi^\epsilon(\cdot,e)$ tends to $0$ at infinity, uniformly in $e$, and therefore attains its maximum at some point $(\overline x_\epsilon,\overline e_\epsilon)$. The first variation in $x$ and $e$ gives
\[
    \nabla\varphi(\overline x_\epsilon\pm\epsilon\overline e_\epsilon)
    =
    \nabla\varphi(\overline x_\epsilon)\pm\delta_\epsilon \overline e_\epsilon,
\]
where
\[
    \delta_\epsilon
    =
    \frac{1}{2}\left\langle
    \nabla\varphi(\overline x_\epsilon+\epsilon\overline e_\epsilon)
    -
    \nabla\varphi(\overline x_\epsilon-\epsilon\overline e_\epsilon),
    \overline e_\epsilon
    \right\rangle.
\]
By convexity,
\begin{equation}\label{Incremental_ratio_bound}
    \sup_{\mathbb{R}^n\times\mathbb{S}^{n-1}}\varphi^\epsilon
    =
    \varphi^\epsilon(\overline x_\epsilon,\overline e_\epsilon)
    \le 2\epsilon\delta_\epsilon.
\end{equation}

The same discrete second-variation computation as in \cite{Alessio_Carlier_Santa}, applied to the Monge-Amp\`ere equation at the three points $\overline x_\epsilon$, $\overline x_\epsilon+\epsilon\overline e_\epsilon$, and $\overline x_\epsilon-\epsilon\overline e_\epsilon$, gives the incremental version of \cref{Exact_second_variation}. If, along a sequence $\epsilon\downarrow0$, the points $\overline x_\epsilon$ remain bounded, then, up to subsequences, $\overline x_\epsilon\to\overline x$ and $\overline e_\epsilon\to e_1$, and the incremental inequality converges to the maximum-principle inequality above.

It remains only to rule out escape of the maximizing points. In the finite case, this follows from \cref{Finite_D_growth}: the map has at most linear growth, so the incremental inequality at infinity contains the same coercive term as in \cite[Theorem 3.2]{Alessio_Carlier_Santa}. Consequently $\delta_\epsilon=o(\epsilon)$ if $|\overline x_\epsilon|\to+\infty$, contradicting \cref{Incremental_ratio_bound} unless $D^2\varphi\equiv0$. When $D=+\infty$, for finite $R$, one works on $B_{R-\eta}(0)$, applies the same argument with interior maxima, and lets $\eta\downarrow0$; the case $R=+\infty$ then follows by letting $R\to+\infty$. This proves the stated bounds, and the truncation and regularization parameters are then sent to their limits.

\begingroup
\emergencystretch=3em
\printbibliography
\endgroup
\end{document}

%% file: comands.tex
\usepackage[T1]{fontenc} 
\usepackage{lmodern} 
\rmfamily 
\DeclareFontShape{T1}{lmr}{b}{sc}{<->ssub*cmr/bx/sc}{}
\DeclareFontShape{T1}{lmr}{bx}{sc}{<->ssub*cmr/bx/sc}{}
\usepackage{lipsum}
\usepackage{mathtools}
\usepackage{amsmath}
\usepackage{amssymb}
\usepackage{tikz}
\usetikzlibrary{matrix,positioning,decorations.pathreplacing,calc}
\usetikzlibrary{
decorations.text,%
decorations.markings,%
shadows}
\usepackage{adjustbox}
\usepackage{graphicx} 

\usepackage{caption}
\usepackage{subcaption}

\usepackage{dsfont} 
\usepackage[format=hang]{caption}

\usepackage[normalem]{ulem}

\usepackage{bm}

\usepackage[
sorting=nty,
maxnames=99,
maxalphanames=5,
natbib=true,
sortcites]{biblatex}

\DeclareNameAlias{default}{family-given}

\graphicspath{{figures/}}

\AtEveryBibitem{
 \clearfield{url}
 \clearfield{issn}
 \clearfield{isbn}
 \clearfield{urldate}
 
 \ifentrytype{book}{
  \clearfield{pages}}{%
 }
}

\usepackage{xargs}                      
\usepackage[prependcaption,textsize=tiny,textwidth=3cm]{todonotes}
\newcommandx{\unsure}[2][1=]{\todo[linecolor=red,backgroundcolor=red!25,bordercolor=red,#1]{#2}}
\newcommandx{\change}[2][1=]{\todo[linecolor=blue,backgroundcolor=blue!25,bordercolor=blue,#1]{#2}}
\newcommandx{\info}[2][1=]{\todo[linecolor=OliveGreen,backgroundcolor=OliveGreen!25,bordercolor=OliveGreen,#1]{#2}}
\newcommandx{\improvement}[2][1=]{\todo[linecolor=black,backgroundcolor=black!25,bordercolor=black,#1]{#2}}
\newcommandx{\thiswillnotshow}[2][1=]{\todo[disable,#1]{#2}}
\usepackage{amsthm}
\usepackage[noabbrev,capitalize]{cleveref}
\crefname{equation}{}{}

\newtheorem{theorem}{Theorem}[section]
\newtheorem{lemma}[theorem]{Lemma}
\newtheorem{proposition}[theorem]{Proposition}
\newtheorem{corollary}[theorem]{Corollary}

\theoremstyle{definition}

\newtheorem{remark}[theorem]{Remark}

\crefname{assumption}{Assumption}{Assumptions}
\crefname{definition}{Definition}{Definitions}
\crefname{corollary}{Corollary}{Corollaries}
\crefname{enumi}{item}{items}

\usepackage{fancyhdr}

\fancypagestyle{plain}{%
\fancyhead[C]{}
\fancyfoot[C]{}

}
\fancypagestyle{nosection}{%
\fancyhead[CE]{}
\fancyhead[CO]{}
\fancyfoot[CE,CO]{\thepage}
}

\DeclareMathOperator{\R}{\mathbb{R}}

\renewcommand{\tilde}{\widetilde}

\renewcommand{\qedsymbol}{$\blacksquare$}

\usepackage{fancyhdr}

\fancypagestyle{plain}{%
\fancyhead[C]{}
\fancyfoot[C]{}

}
\fancypagestyle{nosection}{%
\fancyhead[CE]{}
\fancyhead[CO]{}
\fancyfoot[CE,CO]{\thepage}
}

\renewcommand{\epsilon}{\varepsilon}

\renewcommand{\tilde}{\widetilde}

\usepackage{tcolorbox}
\usetikzlibrary{tikzmark,calc}

\usepackage{enumerate}
\usepackage{upgreek}

%% file: bibliography.bib
@article {Alessio_Carlier_Santa,
    AUTHOR = {Carlier, Guillaume and Figalli, Alessio and Santambrogio,
              Filippo},
     TITLE = {On optimal transport maps between {$\frac{1}{d}$}-concave
              densities},
   JOURNAL = {Ann. Inst. H. Poincar\'e{} C Anal. Non Lin\'eaire},
  FJOURNAL = {Annales de l'Institut Henri Poincar\'e{} C. Analyse Non
              Lin\'eaire},
    VOLUME = {43},
      YEAR = {2026},
    NUMBER = {2},
     PAGES = {483--500},
      ISSN = {0294-1449,1873-1430},
   MRCLASS = {35J96 (35B65)},
  MRNUMBER = {5032965},
       DOI = {10.4171/aihpc/148},
       URL = {https://doi.org/10.4171/aihpc/148},
}

@article{Fathi,
    author = {Fathi, Max},
    title = {Growth estimates on optimal transport maps via
concentration inequalities},
    journal = {Ann. Fac. Sci. Toulouse Math.},
    year = {2026}
}

@article {Brenier,
    AUTHOR = {Brenier, Yann},
     TITLE = {Polar factorization and monotone rearrangement of
              vector-valued functions},
   JOURNAL = {Comm. Pure Appl. Math.},
  FJOURNAL = {Communications on Pure and Applied Mathematics},
    VOLUME = {44},
      YEAR = {1991},
    NUMBER = {4},
     PAGES = {375--417},
      ISSN = {0010-3640,1097-0312},
   MRCLASS = {46E40 (35Q99 46E99 49Q99)},
  MRNUMBER = {1100809},
MRREVIEWER = {Robert\ McOwen},
       DOI = {10.1002/cpa.3160440402},
       URL = {https://doi.org/10.1002/cpa.3160440402},
}

@article {McCann,
    AUTHOR = {McCann, Robert J.},
     TITLE = {Existence and uniqueness of monotone measure-preserving maps},
   JOURNAL = {Duke Math. J.},
  FJOURNAL = {Duke Mathematical Journal},
    VOLUME = {80},
      YEAR = {1995},
    NUMBER = {2},
     PAGES = {309--323},
      ISSN = {0012-7094,1547-7398},
   MRCLASS = {49Q20 (60E05)},
  MRNUMBER = {1369395},
MRREVIEWER = {Carlos\ Matr\'an},
       DOI = {10.1215/S0012-7094-95-08013-2},
       URL = {https://doi.org/10.1215/S0012-7094-95-08013-2},
}

@article {Benamou_Carlier,
    AUTHOR = {Benamou, Jean-David and Carlier, Guillaume and Cuturi, Marco
              and Nenna, Luca and Peyr\'e, Gabriel},
     TITLE = {Iterative {B}regman projections for regularized transportation
              problems},
   JOURNAL = {SIAM J. Sci. Comput.},
  FJOURNAL = {SIAM Journal on Scientific Computing},
    VOLUME = {37},
      YEAR = {2015},
    NUMBER = {2},
     PAGES = {A1111--A1138},
      ISSN = {1064-8275,1095-7197},
   MRCLASS = {49Q20 (47J25 65K05)},
  MRNUMBER = {3340204},
MRREVIEWER = {Shawn\ Xianfu\ Wang},
       DOI = {10.1137/141000439},
       URL = {https://doi.org/10.1137/141000439},
}

@article {Figalli_Maggi_Pratelli,
    AUTHOR = {Figalli, A. and Maggi, F. and Pratelli, A.},
     TITLE = {A mass transportation approach to quantitative isoperimetric
              inequalities},
   JOURNAL = {Invent. Math.},
  FJOURNAL = {Inventiones Mathematicae},
    VOLUME = {182},
      YEAR = {2010},
    NUMBER = {1},
     PAGES = {167--211},
      ISSN = {0020-9910,1432-1297},
   MRCLASS = {49Q20 (35J96)},
  MRNUMBER = {2672283},
MRREVIEWER = {Lorenzo\ Brasco},
       DOI = {10.1007/s00222-010-0261-z},
       URL = {https://doi.org/10.1007/s00222-010-0261-z},
}

@article {Ambrosio_Stra,
    AUTHOR = {Ambrosio, Luigi and Stra, Federico and Trevisan, Dario},
     TITLE = {A {PDE} approach to a 2-dimensional matching problem},
   JOURNAL = {Probab. Theory Related Fields},
  FJOURNAL = {Probability Theory and Related Fields},
    VOLUME = {173},
      YEAR = {2019},
    NUMBER = {1-2},
     PAGES = {433--477},
      ISSN = {0178-8051,1432-2064},
   MRCLASS = {60D05 (49J55 60H15)},
  MRNUMBER = {3916111},
       DOI = {10.1007/s00440-018-0837-x},
       URL = {https://doi.org/10.1007/s00440-018-0837-x},
}

@article {Caff,
    AUTHOR = {Caffarelli, Luis A.},
     TITLE = {Monotonicity properties of optimal transportation and the
              {FKG} and related inequalities},
   JOURNAL = {Comm. Math. Phys.},
  FJOURNAL = {Communications in Mathematical Physics},
    VOLUME = {214},
      YEAR = {2000},
    NUMBER = {3},
     PAGES = {547--563},
      ISSN = {0010-3616,1432-0916},
   MRCLASS = {60E15 (35R35 82B31)},
  MRNUMBER = {1800860},
MRREVIEWER = {Ludger\ R\"uschendorf},
       DOI = {10.1007/s002200000257},
       URL = {https://doi.org/10.1007/s002200000257},
}

@article {Kim_Milman,
    AUTHOR = {Kim, Young-Heon and Milman, Emanuel},
     TITLE = {A generalization of {C}affarelli's contraction theorem via
              (reverse) heat flow},
   JOURNAL = {Math. Ann.},
  FJOURNAL = {Mathematische Annalen},
    VOLUME = {354},
      YEAR = {2012},
    NUMBER = {3},
     PAGES = {827--862},
      ISSN = {0025-5831,1432-1807},
   MRCLASS = {35K20 (35B50 37C10)},
  MRNUMBER = {2983070},
MRREVIEWER = {Daniela\ Ro\c su},
       DOI = {10.1007/s00208-011-0749-x},
       URL = {https://doi.org/10.1007/s00208-011-0749-x},
}

@article {Jhaveri_Colombo,
    AUTHOR = {Colombo, Maria and Figalli, Alessio and Jhaveri, Yash},
     TITLE = {Lipschitz changes of variables between perturbations of
              log-concave measures},
   JOURNAL = {Ann. Sc. Norm. Super. Pisa Cl. Sci. (5)},
  FJOURNAL = {Annali della Scuola Normale Superiore di Pisa. Classe di
              Scienze. Serie V},
    VOLUME = {17},
      YEAR = {2017},
    NUMBER = {4},
     PAGES = {1491--1519},
      ISSN = {0391-173X,2036-2145},
   MRCLASS = {49Q20 (26D10 35J96)},
  MRNUMBER = {3752535},
MRREVIEWER = {Nunzia\ Gavitone},
}

@article {Vald,
    AUTHOR = {Valdimarsson, Stef\'an Ingi},
     TITLE = {On the {H}essian of the optimal transport potential},
   JOURNAL = {Ann. Sc. Norm. Super. Pisa Cl. Sci. (5)},
  FJOURNAL = {Annali della Scuola Normale Superiore di Pisa. Classe di
              Scienze. Serie V},
    VOLUME = {6},
      YEAR = {2007},
    NUMBER = {3},
     PAGES = {441--456},
      ISSN = {0391-173X,2036-2145},
   MRCLASS = {49Q20 (28A35)},
  MRNUMBER = {2370268},
}

@article {Colombo_Fathi,
    AUTHOR = {Colombo, Maria and Fathi, Max},
     TITLE = {Bounds on optimal transport maps onto log-concave measures},
   JOURNAL = {J. Differential Equations},
  FJOURNAL = {Journal of Differential Equations},
    VOLUME = {271},
      YEAR = {2021},
     PAGES = {1007--1022},
      ISSN = {0022-0396,1090-2732},
   MRCLASS = {49Q22 (46N10 60A10)},
  MRNUMBER = {4154935},
MRREVIEWER = {Ugo\ C.\ Bessi},
       DOI = {10.1016/j.jde.2020.09.032},
       URL = {https://doi.org/10.1016/j.jde.2020.09.032},
}

@article {DePHilippis_Shenfeld,
    AUTHOR = {De Philippis, Guido and Shenfeld, Yair},
     TITLE = {Optimal transport maps, majorization, and log-subharmonic
              measures},
   JOURNAL = {Ann. H. Lebesgue},
  FJOURNAL = {Annales Henri Lebesgue},
    VOLUME = {8},
      YEAR = {2025},
     PAGES = {925--963},
      ISSN = {2644-9463},
   MRCLASS = {49Q22 (39B62)},
  MRNUMBER = {5005388},
       DOI = {10.5802/ahl.251},
       URL = {https://doi.org/10.5802/ahl.251},
}

@unpublished{Kolesnikov,
    author = {Kolesnikov, A. V.},
    title = {Mass transportation and contractions},
    YEAR = {2011}, 
    JOURNAL={arXiv},
    EPRINT= {1103.1479},
    
}
